\documentclass[11pt, oneside]{article}   	
\usepackage{geometry}                		
\usepackage{graphicx}				
\usepackage[dvipsnames]{xcolor}	
\usepackage{amssymb, amsmath, amscd, amsthm}
\usepackage{caption}
\usepackage{pgfplots}  
\usepackage{verbatim} 
\usepackage{bm}
\usepackage{mdframed} 
\usepackage{booktabs} 
\usetikzlibrary{3d,calc,positioning,arrows,arrows.meta,shapes.geometric,patterns,scopes}
\definecolor{darkgreen}{RGB}{0,100,0}
\definecolor{tan}{RGB}{210,180,140}
\definecolor{darkbrown}{RGB}{140,70,20}

\makeatletter
\tikzoption{canvas is plane}[]{\@setOxy#1}
\def\@setOxy O(#1,#2,#3)x(#4,#5,#6)y(#7,#8,#9)%
  {\def\tikz@plane@origin{\pgfpointxyz{#1}{#2}{#3}}%
   \def\tikz@plane@x{\pgfpointxyz{#4}{#5}{#6}}%
   \def\tikz@plane@y{\pgfpointxyz{#7}{#8}{#9}}%
   \tikz@canvas@is@plane
  }
\makeatother

\newtheorem{thm}{Theorem}
\newtheorem{prop}{Proposition}
\newtheorem{lemma}{Lemma}
\newtheorem{cor}{Corollary}

\theoremstyle{definition}
\newtheorem*{rmk}{Remark}

\def\N{{\mathbb N}}

\title{A New Take on Classic `Pen Problems'}
\author{David A. Nash} 
\date{Dec.\ 11, 2019}							

\begin{document}
\maketitle
\begin{abstract}
In this article we generalize the classic ``farm pen'' optimization problem from a first course in calculus in a handful of different ways.  We describe the solution to an $n$-dimensional rectangular variant, and then study the situation when the pens are either regular polygons or platonic solids.
\end{abstract}


\begin{mdframed}[backgroundcolor=black!10]
\begin{quote}
``43. A rectangular stockade is to be built which must have a certain area.  If a stone wall already constructed is available for one of the sides, find the dimensions which would make the cost of construction the least." \cite[pg. 134]{Granville}


``A farmer wants to build four fenced enclosures on his farm for his free-range ostriches.  To keep costs down, he is always interested in enclosing as much area as possible with a given amount of fence.  For the fencing projects in Exercises 35--38, determine how to set up each ostrich pen so that the maximum possible area is enclosed, and find this maximum area. ... 37. A rectangular ostrich pen built with 1000 feet of fencing material, divided into three equal sections by two interior fences that run parallel to the exterior side fences as shown next at the left." \cite[pg. 288]{TK}
\end{quote}
\end{mdframed}

``Pen Problems," such as the two given above, have been around in mathematics textbooks for over one hundred years.  The first, from Granville's calculus textbook (published in 1904), was the oldest that the author was able to dig up.  However, several significantly older textbooks have similar style three-dimensional problems involving the construction of rectangular boxes with no tops, thus, it seems likely that there are older examples of two-dimensional pen problems out there as well.  The second example is from Taalman and Kohn's much more modern entry (2014), and it includes three other variations on this same theme.

In each of these problems, we are tasked with maximizing the space we can enclose given some fixed constraint on the shape and size of the border.  In a first course in calculus, students are often taught to begin these problems by solving for one of the variables in the boundary constraint equation in order to reduce the measure of the space enclosed to a function of just one variable.  For example, we might visualize Taalman and Kohn's example from above as in Figure~\ref{fig:TK}.  We would then like to maximize the area $A = (3x)y$ given a fixed perimeter $1000 = 6x + 4y$.

\begin{figure}[ht!]
$$
\begin{tikzpicture}[scale=0.8]
    \draw[ultra thick] (0,0) -- (0,2) -- (1.8,2) -- (1.8,0) -- cycle;
    \draw[ultra thick] (0.6,0) -- (0.6,2);
    \draw[ultra thick] (1.2,0) -- (1.2,2);
    \draw[thick] (-0.3,1) node {$y$};
    \draw[thick] (0.3,-0.3) node {$x$};
    \draw[thick] (0.9,-0.3) node {$x$};
    \draw[thick] (1.5,-0.3) node {$x$};
\end{tikzpicture}
$$
\caption{Three ostrich pens of equal area.}
\label{fig:TK}
\end{figure}
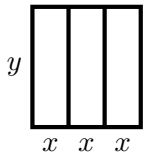

Given this particular setup, one can show the area is maximized when $x = 83.\bar{3}$ and $y = 125$.  What is perhaps more interesting though, is that the amount of fencing used in the vertical direction ($4y = 500$) and the amount used in the horizontal direction ($6x = 500$) are both exactly half of the total fencing allotted.  The same phenomenon occurs in Grenville's version, where -- given a fixed area to enclose -- the perimeter is minimized by letting the portion parallel to the existing wall be twice as long as the sides in the perpendicular direction (hence half the new fencing used is parallel to the wall, and half is perpendicular to it).  In fact, the same is true of the three-dimensional box problem as well -- the volume is maximized when the surface area perpendicular to each cardinal direction is exactly $\frac{1}{3}$ of the total available.



\begin{figure}[ht!]
\[
\begin{tikzpicture}[scale=0.6]
  \pgfmathsetmacro{\cubex}{6}
  \pgfmathsetmacro{\cubey}{2}
  \pgfmathsetmacro{\cubez}{2}
  \draw[thick,black,fill=darkbrown] (0,0,-\cubez) -- ++(-\cubex,0,0) -- ++(0,-\cubey,0) -- ++(\cubex,0,0) -- cycle;
  \draw[thick,black,fill=tan] (0,1,-\cubez) -- ++(-\cubex,0,0) -- ++(0,-1,0) -- ++(\cubex,0,0) -- cycle;
  \draw[thick,black,fill=darkbrown] (-\cubex,0,0) -- ++(0,0,-\cubez) -- ++(0,-\cubey,0) -- ++(0,0,\cubez) -- cycle;
  \draw[thick,black,fill=darkbrown] (-\cubex,\cubey,0) -- ++(0,0,-\cubez) -- ++(0,-\cubey,0) -- ++(0,0,\cubez) -- cycle;
  \draw[thick,black,fill=darkbrown] (0,-\cubey,0) -- ++(0,0,-\cubez) -- ++(-\cubex,0,0) -- ++(0,0,\cubez) -- cycle;
  \draw[thick,black,fill=darkbrown] (0,0,0) -- ++(-\cubex,0,0) -- ++(0,0,-\cubez) -- ++(\cubex,0,0) -- cycle;
  \draw[thick,black,fill=darkbrown] (0,\cubey,0) -- ++(0,0,-\cubez) -- ++(-\cubex,0,0) -- ++(0,0,\cubez) -- cycle;
  \draw[thick,black,fill=darkbrown] (0,0,0) -- ++(0,0,-\cubez) -- ++(0,-\cubey,0) -- ++(0,0,\cubez) -- cycle;
  \draw[thick,black,fill=darkbrown] (0,\cubey,0) -- ++(0,0,-\cubez) -- ++(0,-\cubey,0) -- ++(0,0,\cubez) -- cycle;
  \draw[thick] (-3,-2.5,0) node {$x$};
  \draw[thick] (0.5,-2,-0.5) node {$y$};
  \draw[thick] (0.3, -1, -2) node {$z$};
  \draw[thick] (0.3, 1, -2) node {$z$};
\end{tikzpicture}
\]
\caption{A two-tiered tv stand made out of wood and particleboard.}
\label{fig:tvstand}
\end{figure}
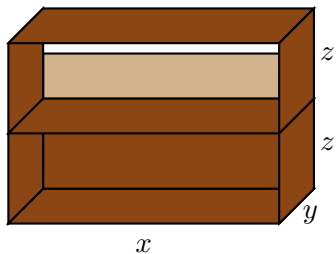

Another common variant involves designs using different materials for some of the walls (each with different costs associated to them).  For example, we might design a tv stand (see Figure~\ref{fig:tvstand}) as two rectangular spaces which are open at the front.  Most of the walls will be made out of wood (at \$3 per sqft), but the back wall of the top section will be half particle board (at \$2 per sqft) and half open (to allow cables to pass out the back).  Given these design choices, we'd like to maximize the volume we can hold (regardless of the aesthetics) given a fixed \emph{cost} of \$81 (rather than a fixed surface area).


Under these conditions, the volume we'd like to maximize is $V = xy(2z)$, and the cost constraint is $81 = 3(3)(xy) + 4(3)(yz) + [1(3) + \frac{1}{2}(2)](xz) = 9xy + 12yz + 4xz$.  We'll again spare the details, but the volume can be maximized here by taking $x=3$, $y=1$, and $z=\frac{9}{4}$.  More importantly, observe that the cost of the horizontal walls (those perpendicular to the $z$-axis) is exactly $9(3)(1) = \$27$, the cost of the walls perpendicular to the $y$-axis is $12(1)(\frac{9}{4}) = \$27$, and the cost of the walls perpendicular to the $x$-axis is $4(3)(\frac{9}{4}) = \$27$.  So we maximize our volume by splitting the cost evenly among each of the directions!

Does the optimal solution always correspond to splitting the constraint equally among each of the directions? In Section~\ref{sec:rect}, we begin by addressing this question in a rectangular setup that has been generalized to any number of dimensions.
We follow that up in Sections~\ref{sec:poly} and \ref{sec:packing} by considering chains of pens made out of equal sized regular polygons and more compact arrangements of triangles, squares, and hexagons.  Finally, in Section~\ref{sec:platonic}, we consider the equivalent 3-dimensional situation with chains of equal sized platonic solid pens.  At each stage, we also compare to using circles or spheres where it is impossible to share any of the boundary.

\section{The $\bm{n}$-dimensional Rectangular Pen Problem}\label{sec:rect}
In $n$-dimensions ($n \geq 2$), given positive integers $b_1, \dots, b_n \in \N$, we might create a $b_1 \times b_2 \times \cdots \times b_n$ grid of equal size rectangular $n$-dimensional spaces each with identical side lengths $x_1, \dots, x_n \geq 0$.  Thus, the hypervolume that we'd like to maximize is
$$V = \prod_{j=1}^n (x_jb_j)$$

We will consider the case in which our constraint is a fixed cost, $C$, under the assumption that the cost of each possible material per unit of hypersurface area is constant.  Observe,  if we focus only on walls that are perpendicular to the $i$-th direction, a single wall for a single chamber has hypersurface area $\prod_{j \neq i} x_j$.  Perhaps some of the chamber walls already exist (or partially exist), or we may design the grid so that only certain ones use particular materials, etc.  Regardless of the particular features, we may count the total cost of these walls by adding up a linear combination of the material costs with weights coming from a \emph{count} (possibly fractional) of the number of walls of each type -- just as we observed in the tv stand example above.

Importantly, this is a linear combination of various constants that have been established within the design parameters, thus we may represent that linear combination with a single constant $c_i$.  For simplicity in what follows, we set $C_i = \frac{c_i}{\prod_{j \neq i} b_j}$ so that cost of all walls perpendicular to the $i$-th direction is $C_i \prod_{j \neq i} (x_jb_j)$ and the total cost is exactly:
$$C = \sum_{i=1}^n C_i \prod_{j \neq i} (x_jb_j)$$

\begin{lemma}\label{lem:Ci>0} If $C_i = 0$ for any $i$, then the volume $V$ can be made arbitrarily large.
\end{lemma}
\begin{proof} Without loss of generality, we assume $C_n = 0$, then every term in our constraint equation contains a factor of $x_n$.  If we take $x_i= x_1$ for all $i < n$, and 
$$x_n = \frac{C}{\sum_{i<n} C_i b_n \prod_{j \neq i,n} (x_j b_j)} = \frac{C}{x_1^{n-2} \sum_{i < n} C_i \prod_{j \neq i} b_j}.$$
This will satisfy the constraint equation for all possible $x_1 > 0$ and will have volume
$$V = \frac{x_1^{n-1} \prod_{j=1}^n b_j \cdot C}{x_1^{n-2} \sum_{i<n} C_i \prod_{j \neq i} b_j},$$
which is proportional to $x_1$ and thus can be made arbitrarily large.
\end{proof}

Given Lemma~\ref{lem:Ci>0}, in what follows we will assume that $C_i > 0$ for all $i$ so that a finite maximum volume exists.

\begin{thm} When constructing a $b_1 \times b_2 \times \dots \times b_n$ grid of rectangular $n$-dimensional spaces subject to a fixed cost, $C$, the hypervolume is maximized when the cost of the walls perpendicular to each direction is exactly $\frac{C}{n}$.
\end{thm}
\begin{proof}
Recall, from the method of Lagrange multipliers, we know that the maximum volume, subject to our fixed cost constraint, should occur when, for all $1 \leq i \leq n$, we have $\frac{\partial V}{\partial x_i} = \lambda \frac{\partial C}{\partial x_i}$ for some parameter $\lambda \neq 0$.  Thus, for any such $i$, we have  
$$\frac{\partial V}{\partial x_i} = \lambda \frac{\partial C}{\partial x_i} \implies b_i\prod_{j \neq i} (x_j b_j) = \lambda \left[\sum_{k \neq i} \left( b_i C_k \prod_{j \neq k, i} (x_j b_j) \right)\right]$$
Since $V$ is zero whenever $x_i=0$ for any $i$, we will ignore these potential solutions and assume that $x_i > 0$ for all $i$.  Thus, we may divide both sides by $b_i \prod_{j \neq i} (x_j b_j)$ to obtain $1 = \lambda \sum_{k \neq i} \frac{C_k}{x_kb_k}$ for each $i$.  Equating the sums for two different indices, $i \neq j$, we have $\lambda \sum_{k \neq i} \frac{C_k}{x_kb_k} = \lambda \sum_{k \neq j} \frac{C_k}{x_kb_k}$ which implies that $\frac{C_j}{x_j b_j} = \frac{C_i}{x_i b_i}$ for all $i$ and $j$.  Thus, after substituting, our equation becomes $1 = \lambda (n-1) \frac{C_i}{x_i b_i}$ for all $i$.  Hence, $x_ib_i = \lambda (n-1) C_i$ for all $i$ and we may substitute this fact into our cost equation to obtain
%
%
%
%
%
%
$$C = \sum_{i=1}^n C_i \prod_{j \neq i} \lambda(n-1)C_j = n \lambda^{n-1} (n-1)^{n-1} \prod_{i=1}^n C_i.$$
Moreover, isolating the cost of the walls that are perpendicular to the $i$-th direction gives $C_i \prod_{j \neq i} (a_j b_j) = C_i \prod_{j \neq i} \lambda (n-1) C_j = \lambda^{n-1} (n-1)^{n-1} C_1 \cdots C_n$, which is exactly 
$\frac{C}{n}$.
\end{proof}

\begin{cor}[Surface Area]\label{cor:rect}
When constructing a $b_1 \times b_2 \times \dots \times b_n$ grid of rectangular $n$-dimensional spaces subject to a fixed amount of hypersurface area, $S$, the hypervolume is maximized when the hypersurface area perpendicular to each direction is exactly $\frac{S}{n}$.
\end{cor}
\begin{proof} $(b_i+1)$ counts the number of full-length walls -- each with hypersurface area $\prod_{j \neq i} (a_j b_j)$ -- that are perpendicular to the $i$-th direction.  Thus, setting $C_i = (b_i+1)$ converts the cost equation to an equation for the total hypersurface area.
\end{proof}



\section{Regular Polygon Pens}\label{sec:poly}
If the goal is to maximize the enclosed space, we might want to consider other kinds of designs as well -- for example, those that use non-rectangular shapes.  Back in two dimensions, it is well-known that a single circular enclosure will give you the best ratio of area to perimeter.  However, in creating multiple pens, those circular designs won't be able to share sides, which seems counterproductive.  More precisely, given a fixed perimeter $P$ to work with, the perimeter of $k$ circular pens must satisfy $P = 2\pi r k$, which implies that the enclosed area is 
$$A_k(\infty) = k\pi \left(\frac{P}{2\pi k}\right)^2  = \frac{P^2}{4\pi k}.$$

If we choose to use regular $n$-sided polygons instead, then we can always create a chain of $k$ of these polygons so that each shares a side with its neighbors, see Figure~\ref{fig:pentagons}.  Recall that for a single pen, the ratio of area to perimeter grows with $n$.  However, with multiple pens, the amount of shared perimeter shrinks as $n$ grows.  Our goal is thus to determine the best balancing point for these counteractive forces. 

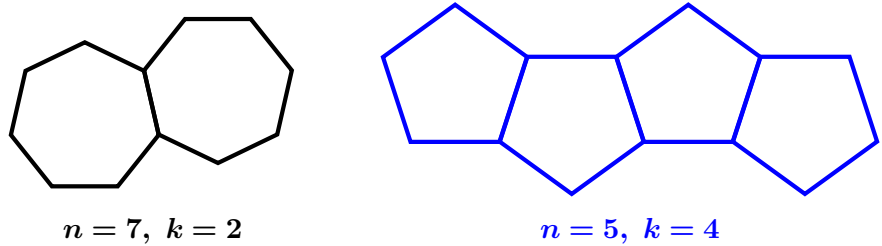
\begin{figure}[ht!]
$$
\begin{tikzpicture}
    \node[draw=none,minimum size=2cm,regular polygon,regular polygon sides=7] (a) {};
    \draw[ultra thick] (a.corner 1) -- (a.corner 2) -- (a.corner 3) -- (a.corner 4) -- (a.corner 5) -- (a.corner 6) -- (a.corner 7) -- cycle;
    \node[draw=none,minimum size=2cm,regular polygon,regular polygon sides=7,rotate=25.75](b){};
    \draw[ultra thick] let \p{a}=(b.corner 1), \p{b}=(b.corner 2), \p{c}=(b.corner 3), \p{d}=(b.corner 4), \p{e}=(b.corner 5), \p{f}=(b.corner 6), \p{g}=(b.corner 7) in
    {[shift={(1.77,0.41)}] (\p{a}) -- (\p{b}) -- (\p{c}) -- (\p{d}) -- (\p{e}) -- (\p{f}) -- (\p{g}) -- cycle};
    \draw[ultra thick] (0.9,-1.5) node {$\bm{n=7,~ k=2}$};
\end{tikzpicture}
\qquad \quad
\begin{tikzpicture}
    \node[draw=none,minimum size=2cm,regular polygon,regular polygon sides=5] (a) {};
    \draw[ultra thick, blue] (a.corner 1) -- (a.corner 2) -- (a.corner 3) -- (a.corner 4) -- (a.corner 5) -- cycle;
    \node[draw=none,minimum size=2cm,regular polygon,regular polygon sides=5,rotate=36](b){};
    \draw[ultra thick, blue] let \p{a}=(b.corner 1), \p{b}=(b.corner 2), \p{c}=(b.corner 3), \p{d}=(b.corner 4), \p{e}=(b.corner 5) in
    {[shift={(1.55,-0.5)}] (\p{a}) -- (\p{b}) -- (\p{c}) -- (\p{d}) -- (\p{e}) -- cycle}
    {[shift={(4.65,-0.5)}] (\p{a}) -- (\p{b}) -- (\p{c}) -- (\p{d}) -- (\p{e}) -- cycle};
    \draw[ultra thick, blue] let \p{a}=(a.corner 1), \p{b}=(a.corner 2), \p{c}=(a.corner 3), \p{d}=(a.corner 4), \p{e}=(a.corner 5) in
    {[shift={(3.1,0)}] (\p{a}) -- (\p{b}) -- (\p{c}) -- (\p{d}) -- (\p{e}) -- cycle};
    \draw[ultra thick, blue] (2.325,-2) node {$\bm{n=5,~ k=4}$};
\end{tikzpicture}
$$
\caption{A chain of two heptagons, and a chain of four pentagons.}
\label{fig:pentagons}
\end{figure}

It is well-known \cite{Johnson} that a regular $n$-sided polygon with side length $s$ has area $\frac{1}{4}ns^2 \cot(\frac{\pi}{n})$ and a perimeter of $ns$.  Thus, if we create a chain of $k$ regular $n$-gons, 
then the total area and perimeter are exactly:
$$A = \frac{k}{4} \cdot ns^2\cot(\frac{\pi}{n}) \quad \text{and} \quad P = kns - (k-1)s = s(k(n-1) +1)$$
Solving for $s$ in the perimeter equation and substituting, we find that (for a fixed integer $k \in \N$), the total area of this arrangement is
$$A_k(n) = \frac{P^2 k}{4} \cdot \frac{n \cot(\frac{\pi}{n})}{[k(n-1)+1]^2} \quad (n \geq 3).$$
In other words, given a fixed perimeter, $P$, with which to construct these pens, the area depends only on which polygon we choose.

If we consider $A_k(x)$ as a continuous function ($x \geq 3$) rather than a discrete one, we may take the derivative using the quotient rule.  After multiplying the numerator and denominator by $\sin^2(\frac{\pi}{x})$ and applying the double angle identity, we have:
$$A_k'(x) = \frac{P^2k}{4} \cdot \frac{\frac{1}{2}\sin(\frac{2\pi}{x})[1 - k(x+1)] + \frac{\pi}{x}[k(x-1)+1]}{\sin^2(\frac{\pi}{x})[k(x-1)+1]^3}$$
Since $x \geq 3$, the sign of this derivative, for each $k \in \N$, is completely determined by the sign of $N_k(x):=\frac{1}{2}\sin(\frac{2\pi}{x})[1 - k(x+1)] + \frac{\pi}{x}[k(x-1)+1]$.  This can be rearranged by putting all terms with a $k$ together to get:
$$N_k(x) = k\left[\frac{(n-1)\pi}{x} - \frac{x+1}{2}\sin\left(\frac{2\pi}{x}\right)\right] + \frac{\pi}{x} + \frac{1}{2}\sin\left(\frac{2\pi}{x}\right)$$

In order to explore $N_k(x)$ it will be helpful to replace the sine terms with rational functions instead.  It is well-known that $\sin(\theta) < \theta$ for all $\theta >0$, thus $\sin(\frac{2\pi}{x}) < \frac{2\pi}{x}$ for all $x > 0$.  More important to our discussion, however, is the following result:

\begin{lemma}\label{lem:bound}
$\sin\left(\frac{2\pi}{x}\right) > \frac{2\pi}{x+1}$ for all $x > 
7.5$.
\end{lemma}
\begin{proof} Recall that the Taylor series representation of sine is $\displaystyle \sum_{i=0}^\infty (-1)^i \frac{x^{2i+1}}{(2i+1)!}$. Hence, using the second Taylor polynomial, $\sin\left(\frac{2\pi}{x}\right) > \frac{2\pi}{x} - \frac{4\pi^3}{3x^3}$ for all $x > 0$.  In addition, observe that $\frac{2\pi}{x+1} = \frac{2\pi}{x} - \frac{2\pi}{x(x+1)}$.  Thus, it will follow that $\sin\left(\frac{2\pi}{x}\right) > \frac{2\pi}{x+1}$ whenever $\frac{4\pi^3}{3x^3} < \frac{2\pi}{x(x+1)}$, or rearranging, when $6\pi x^3 - 4\pi^3x^2 - 4\pi^3 x > 0$ and one can directly demonstrate that this holds for all $x \geq 7.5$. (Note: this inequality is not sharp)
\end{proof}
Applying Lemma~\ref{lem:bound} and the fact that $\sin(\frac{2\pi}{x}) < \frac{2\pi}{x}$, we have:
$$D_k(x) < k\left[\frac{(n-1)\pi}{x} - \frac{x+1}{2} \cdot \frac{2 \pi}{x+1} \right] + \frac{\pi}{x} + \frac{1}{2} \cdot \frac{2\pi}{x} = \frac{\pi(2-k)}{x} \text{~for all $x \geq 7.5$.}$$
Hence, $D_k(x) < 0$ (and therefore, $A_k'(x) < 0$) for all $x \geq 7.5$ and $k \geq 2$.
It follows that, for $k \geq 2$, we have $A_k(n+1) < A_k(n)$ for all $n \geq 8$.  Thus, for multiple pens, any polygons with more than 8 sides will always be worse than options with fewer sides.  In fact, for these polygons, the more sides they have, the less total area is enclosed.  Now, only $A_k(3)$, $A_k(4)$, $A_k(5)$, $A_k(6)$, $A_k(7)$, and $A_k(8)$ are left to consider. 

Observe that $A_k(n)<A_k(m)$ exactly when $\frac{P^2 k}{4} \cdot \frac{n \cot(\frac{\pi}{n})}{[k(n-1)+1]^2} < \frac{P^2 k}{4} \cdot \frac{m \cot(\frac{\pi}{m})}{[k(m-1)+1]^2}$, which can be rearranged as $\frac{k(m-1)+1}{k(n-1)+1} < \sqrt{\frac{m\cot(\frac{\pi}{m})}{n \cot(\frac{\pi}{n})}}$.  As a function of $k$, the rational function on the left is always increasing if $m>n$ and always decreasing if $m<n$, so we need only find the value of $k$ (if any) at which this function passes the constant value to which it is being compared.  

Since we only care about positive integer values of $k$, we will refrain from reporting more exact values for these inequalities.  For example, $A_3(k) < A_4(k)$ exactly when $\frac{3k+1}{2k+1} < \sqrt{\frac{4}{\sqrt{3}}}$ which is true for all $k > \frac{3^{1/4}-2}{4-3^{5/4}}$ ($\approx -13.21$), however, for our purposes it is enough to say $A_3(k) < A_4(k)$ for all $k \geq 1$.  Interpreting this conclusion, no matter how many pens we wish to build, it would always be better to use squares than equilateral triangles.  

Continuing in this way (again, reporting only positive integer values), one can show that $A_k(8) < A_k(7)$ for all $k \geq 1$, $A_k(7) < A_k(6)$ for all $k \geq 3$, that $A_k(6) < A_k(5)$ for all $k \geq 3$, and that $A_k(5) < A_k(4)$ for all $k \geq 5$.  Thus:
\begin{thm} When given a fixed amount of perimeter to create a chain of $k$ equal sized pens out of regular $n$-gons, the area is maximized when using heptagons if $k=2$, pentagons if $k=3$ or $4$, and squares for all $k \geq 5$.
\end{thm}

\begin{rmk}[Circles]
Recall that the limit as $n$ goes to infinity of an $n$-gon is a circle.  Thus, our exploration of the derivative of $A_k(x)$ above demonstrates that, using circles is worse than using any regular polygon with 8 or more sides for all $k \geq 2$.  Using the same technique as above, observe that $\frac{P^2}{4\pi k} = A_k(\infty) < A_k(4) = \frac{P^2k}{4} \cdot \frac{4}{(3k+1)^2}$ exactly when $\sqrt{\frac{1}{4\pi}} < \frac{k}{3k+1}$ which is true for all $k \geq 2$.  Hence, given our comparisons above, using circles is worse than using polygons with 4 or more sides as well for all $k \geq 2$.  Interestingly however, $A_k(\infty) < A_k(3)$ when $\sqrt{\frac{1}{\pi\sqrt{3}}} < \frac{k}{2k+1}$ which is true only for $k \geq 4$.  Thus, it is more efficient to use two or three circles than two or three triangles, in spite of the inability to share any of the perimeter.
\end{rmk}

\section{Pen Packing}\label{sec:packing}
Of course, with equilateral triangles, squares, and regular hexagons we can do better by arranging our pens so that they share multiple sides.  Harary and Harborth \cite{HH} proved, in each of these cases, that using a spiral pattern corresponds to the arrangement of $k$ polygons with the most shared sides, see Figure~\ref{fig:HH}.


\newpage
\begin{figure}[ht!]
$$
\begin{tikzpicture}
    \node[draw=none,minimum size=1cm,regular polygon,regular polygon sides=3] (a) {};
    \draw[thick] (a.corner 1) -- (a.corner 2) -- (a.corner 3) -- cycle;
    \draw[thick] (a.center) node {3};
    \draw[thick] let \p{a}=(a.corner 1), \p{b}=(a.corner 2), \p{c}=(a.corner 3), \p{d}=(a.center) in
    {[shift={(.9,0)}] (\p{a}) -- (\p{b}) -- (\p{c}) -- cycle}
    {[shift={(.9,0)}] (\p{d}) node {5}}
    {[shift={(1.8,0)}] (\p{a}) -- (\p{b}) -- (\p{c}) -- cycle}
    {[shift={(1.8,0)}] (\p{d}) node {23}}
    {[shift={(0.45,-.78)}] (\p{a}) -- (\p{b}) -- (\p{c}) -- cycle}
    {[shift={(0.45,-0.78)}] (\p{d}) node {1}}
    {[shift={(1.35,-0.78)}] (\p{a}) -- (\p{b}) -- (\p{c}) -- cycle}
    {[shift={(1.35,-0.78)}] (\p{d}) node {7}}
    {[shift={(0.9,-1.56)}] (\p{a}) -- (\p{b}) -- (\p{c}) -- cycle}
    {[shift={(0.9,-1.56)}] (\p{d}) node {9}}
    {[shift={(0,-1.56)}] (\p{a}) -- (\p{b}) -- (\p{c}) -- cycle}
    {[shift={(0,-1.56)}] (\p{d}) node {11}}
    {[shift={(-0.45,-.78)}] (\p{a}) -- (\p{b}) -- (\p{c}) -- cycle}
    {[shift={(-0.45,-0.78)}] (\p{d}) node {13}}
    {[shift={(-0.9,0)}] (\p{a}) -- (\p{b}) -- (\p{c}) -- cycle}
    {[shift={(-0.9,0)}] (\p{d}) node {15}}
    {[shift={(-0.45,.78)}] (\p{a}) -- (\p{b}) -- (\p{c}) -- cycle}
    {[shift={(-0.45,0.78)}] (\p{d}) node {17}}
    {[shift={(0.45,.78)}] (\p{a}) -- (\p{b}) -- (\p{c}) -- cycle}
    {[shift={(0.45,0.78)}] (\p{d}) node {19}}
    {[shift={(1.35,.78)}] (\p{a}) -- (\p{b}) -- (\p{c}) -- cycle}
    {[shift={(1.35,0.78)}] (\p{d}) node {21}}
    {[shift={(2.25,-0.78)}] (\p{a}) -- (\p{b}) -- (\p{c}) -- cycle}
    {[shift={(2.25,-0.78)}] (\p{d}) node {25}}
    ;
    \node[draw=none,minimum size=1cm,regular polygon,regular polygon sides=3,rotate=180](b){};
    \draw[thick] let \p{a}=(b.corner 1), \p{b}=(b.corner 2), \p{c}=(b.corner 3), \p{d}=(a.center) in
    {[shift={(0.45,0.26)}] (\p{a}) -- (\p{b}) -- (\p{c}) -- cycle}
    {[shift={(.45,0.26)}] (\p{d}) node {4}}
    {[shift={(1.35,0.26)}] (\p{a}) -- (\p{b}) -- (\p{c}) -- cycle}
    {[shift={(1.35,0.26)}] (\p{d}) node {22}}
    {[shift={(0,-.52)}] (\p{a}) -- (\p{b}) -- (\p{c}) -- cycle}
    {[shift={(0,-0.52)}] (\p{d}) node {2}}
    {[shift={(0.9,-0.52)}] (\p{a}) -- (\p{b}) -- (\p{c}) -- cycle}
    {[shift={(0.9,-0.52)}] (\p{d}) node {6}}
    {[shift={(1.35,-1.3)}] (\p{a}) -- (\p{b}) -- (\p{c}) -- cycle}
    {[shift={(1.35,-1.3)}] (\p{d}) node {8}}
    {[shift={(.45,-1.3)}] (\p{a}) -- (\p{b}) -- (\p{c}) -- cycle}
    {[shift={(.45,-1.3)}] (\p{d}) node {10}}
    {[shift={(-.45,-1.3)}] (\p{a}) -- (\p{b}) -- (\p{c}) -- cycle}
    {[shift={(-.45,-1.3)}] (\p{d}) node {12}}
    {[shift={(-0.9,-.52)}] (\p{a}) -- (\p{b}) -- (\p{c}) -- cycle}
    {[shift={(-0.9,-0.52)}] (\p{d}) node {14}}
    {[shift={(-0.45,0.26)}] (\p{a}) -- (\p{b}) -- (\p{c}) -- cycle}
    {[shift={(-.45,0.26)}] (\p{d}) node {16}}
    {[shift={(0,1.04)}] (\p{a}) -- (\p{b}) -- (\p{c}) -- cycle}
    {[shift={(0,1.04)}] (\p{d}) node {18}}
    {[shift={(0.9,1.04)}] (\p{a}) -- (\p{b}) -- (\p{c}) -- cycle}
    {[shift={(0.9,1.04)}] (\p{d}) node {20}}
    {[shift={(1.8,-0.52)}] (\p{a}) -- (\p{b}) -- (\p{c}) -- cycle}
    {[shift={(1.8,-0.52)}] (\p{d}) node {24}};
    \draw [ultra thick, -latex] (2.25,-1.15) -- (1.85, -1.7);
\end{tikzpicture}
\qquad
\begin{tikzpicture}[scale=0.7]
    \draw[thick] (0,0) -- (4,0) -- (4,4) -- (0,4) -- (0,0) ++ (1,0) -- (1,4) ++ (1,0) -- (2,0) ++ (1,0) -- (3,4) ++ (1,-1) -- (0,3) ++ (0,-1) -- (4,2) ++ (0,-1) -- (0,1);
    \draw[thick] (0.5,0.5) node {10};
    \draw[thick] (1.5,0.5) node {9};
    \draw[thick] (2.5,0.5) node {8};
    \draw[thick] (3.5,0.5) node {7};
    \draw[thick] (0.5,1.5) node {11};
    \draw[thick] (1.5,1.5) node {2};
    \draw[thick] (2.5,1.5) node {1};
    \draw[thick] (3.5,1.5) node {6};
    \draw[thick] (0.5,2.5) node {12};
    \draw[thick] (1.5,2.5) node {3};
    \draw[thick] (2.5,2.5) node {4};
    \draw[thick] (3.5,2.5) node {5};
    \draw[thick] (0.5,3.5) node {13};
    \draw[thick] (1.5,3.5) node {14};
    \draw[thick] (2.5,3.5) node {15};
    \draw[thick] (3.5,3.5) node {16};
    \draw[thick] (4,4) -- (5,4) -- (5,3) -- (4,3);
    \draw[thick] (4.5,3.5) node {17};
    \draw[ultra thick, -latex] (4.5,2.8) -- (4.5, 2);
\end{tikzpicture}
\qquad
\begin{tikzpicture}[scale=0.8]
    \node[draw=none,minimum size=0.8cm,regular polygon,regular polygon sides=6] (a) {};
    \draw[thick] let \p{a}=(a.corner 1), \p{b}=(a.corner 2), \p{c}=(a.corner 3), \p{d}=(a.corner 4), \p{e}=(a.corner 5), \p{f}=(a.corner 6), \p{g}=(a.center), \p{r}=($(a.corner 1) - (a.corner 4)$), \p{u} = ($(a.corner 2) - (a.corner 5)$) in
    {[shift={(0,0)}] (\p{a}) -- (\p{b}) -- (\p{c}) -- (\p{d}) -- (\p{e}) -- (\p{f}) -- cycle}
    {[shift={(0,0)}] (\p{g}) node {1}}
    {[shift={(0.76,-0.44)}] (\p{a}) -- (\p{b}) -- (\p{c}) -- (\p{d}) -- (\p{e}) -- (\p{f}) -- cycle}
    {[shift={(0.76,-0.44)}] (\p{g}) node {2}}
    {[shift={(0,-0.88)}] (\p{a}) -- (\p{b}) -- (\p{c}) -- (\p{d}) -- (\p{e}) -- (\p{f}) -- cycle}
    {[shift={(0,-0.88)}] (\p{g}) node {3}}
    {[shift={(-0.76,-0.44)}] (\p{a}) -- (\p{b}) -- (\p{c}) -- (\p{d}) -- (\p{e}) -- (\p{f}) -- cycle}
    {[shift={(-0.76,-0.44)}] (\p{g}) node {4}}
    {[shift={(-0.76,0.44)}] (\p{a}) -- (\p{b}) -- (\p{c}) -- (\p{d}) -- (\p{e}) -- (\p{f}) -- cycle}
    {[shift={(-0.76,0.44)}] (\p{g}) node {5}}
    {[shift={(0,0.88)}] (\p{a}) -- (\p{b}) -- (\p{c}) -- (\p{d}) -- (\p{e}) -- (\p{f}) -- cycle}
    {[shift={(0,0.88)}] (\p{g}) node {6}}
    {[shift={(0.76,0.44)}] (\p{a}) -- (\p{b}) -- (\p{c}) -- (\p{d}) -- (\p{e}) -- (\p{f}) -- cycle}
    {[shift={(0.76,0.44)}] (\p{g}) node {7}}
    {[shift={(1.52,0)}] (\p{a}) -- (\p{b}) -- (\p{c}) -- (\p{d}) -- (\p{e}) -- (\p{f}) -- cycle}
    {[shift={(1.52,0)}] (\p{g}) node {8}}
    {[shift={(1.52,-0.88)}] (\p{a}) -- (\p{b}) -- (\p{c}) -- (\p{d}) -- (\p{e}) -- (\p{f}) -- cycle}
    {[shift={(1.52,-0.88)}] (\p{g}) node {9}}
    {[shift={(0.76,-1.32)}] (\p{a}) -- (\p{b}) -- (\p{c}) -- (\p{d}) -- (\p{e}) -- (\p{f}) -- cycle}
    {[shift={(0.76,-1.32)}] (\p{g}) node {10}}
    {[shift={(0,-1.76)}] (\p{a}) -- (\p{b}) -- (\p{c}) -- (\p{d}) -- (\p{e}) -- (\p{f}) -- cycle}
    {[shift={(0,-1.76)}] (\p{g}) node {11}}
    {[shift={(-0.76,-1.32)}] (\p{a}) -- (\p{b}) -- (\p{c}) -- (\p{d}) -- (\p{e}) -- (\p{f}) -- cycle}
    {[shift={(-0.76,-1.32)}] (\p{g}) node {12}}
    {[shift={(-1.52,-0.88)}] (\p{a}) -- (\p{b}) -- (\p{c}) -- (\p{d}) -- (\p{e}) -- (\p{f}) -- cycle}
    {[shift={(-1.52,-0.88)}] (\p{g}) node {13}}
    {[shift={(-1.52,0)}] (\p{a}) -- (\p{b}) -- (\p{c}) -- (\p{d}) -- (\p{e}) -- (\p{f}) -- cycle}
    {[shift={(-1.52,0)}] (\p{g}) node {14}}
    {[shift={(-1.52,0.88)}] (\p{a}) -- (\p{b}) -- (\p{c}) -- (\p{d}) -- (\p{e}) -- (\p{f}) -- cycle}
    {[shift={(-1.52,0.88)}] (\p{g}) node {15}}
    {[shift={(-0.76,1.32)}] (\p{a}) -- (\p{b}) -- (\p{c}) -- (\p{d}) -- (\p{e}) -- (\p{f}) -- cycle}
    {[shift={(-0.76,1.32)}] (\p{g}) node {16}}
    {[shift={(0,1.76)}] (\p{a}) -- (\p{b}) -- (\p{c}) -- (\p{d}) -- (\p{e}) -- (\p{f}) -- cycle}
    {[shift={(0,1.76)}] (\p{g}) node {17}}
    {[shift={(0.76,1.32)}] (\p{a}) -- (\p{b}) -- (\p{c}) -- (\p{d}) -- (\p{e}) -- (\p{f}) -- cycle}
    {[shift={(0.76,1.32)}] (\p{g}) node {18}}
    {[shift={(1.52,0.88)}] (\p{a}) -- (\p{b}) -- (\p{c}) -- (\p{d}) -- (\p{e}) -- (\p{f}) -- cycle}
    {[shift={(1.52,0.88)}] (\p{g}) node {19}}
    {[shift={(2.28,0.44)}] (\p{a}) -- (\p{b}) -- (\p{c}) -- (\p{d}) -- (\p{e}) -- (\p{f}) -- cycle}
    {[shift={(2.28,0.44)}] (\p{g}) node {20}}
    ;
    \draw[ultra thick, -latex] (2.28,-.1) -- (2.28, -1);
\end{tikzpicture}
$$
\caption{Spiral arrangements of 25 triangles, 17 squares, and 20 hexagons.}
\label{fig:HH}
\end{figure}
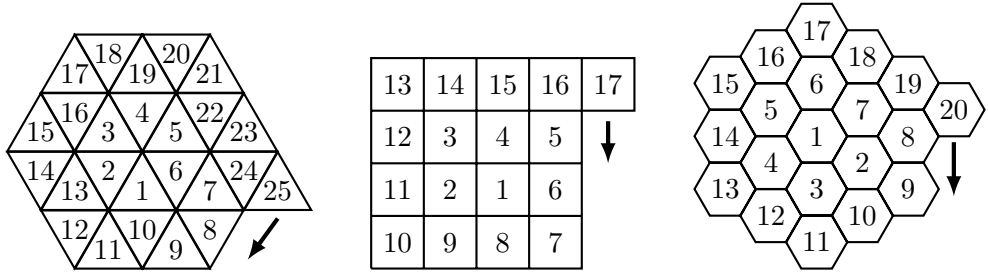

In fact, they provide a formula for the number of sides in each arrangement:
\begin{prop}\emph{\cite[Theorems 2, 3, and 4]{HH}}\label{prop:HH}\\
\emph{(1)} The number of sides in a spiral arrangement of $k$ triangles is $k + \lceil\frac{1}{2}(k + \sqrt{6k})\rceil$.\\
\emph{(2)} The number of sides in a spiral arrangement of $k$ squares is $2k + \lceil2\sqrt{k}\rceil$.\\
\emph{(3)} The number of sides in a spiral arrangement of $k$ hexagons is $3k + \lceil\sqrt{12k-3}\rceil$.
\end{prop}

In the discussion that follows, we'll use $A^\ast_k(n)$ to denote the total area of each spiral arrangement of $k$ triangles ($n=3$), squares ($n=4$), or hexagons ($n=6$).  
Of course, since sharing additional sides is more efficient, it follows that $A^\ast_k(n) \geq A_k(n)$ for each $n$.
By (1), we know that the number of sides for $k$ triangles is always greater than or equal to $\frac{1}{2}(3k + \sqrt{6k})$.  The more sides we have, the shorter they'll each have to be, hence with fixed perimeter $P$, we have side length $s \leq\frac{2P}{3k+\sqrt{6k}}$. This implies that the area of an arrangement of $k$ triangles, which is equal to $k\cdot \frac{\sqrt{3}}{4} s^2$, satisfies 
$$A^\ast_k(3) \leq \frac{\sqrt{3}k}{4} \cdot \left(\frac{2P}{3k+\sqrt{6k}}\right)^2 = \frac{\sqrt{3} \cdot kP^2}{(3k+\sqrt{6k})^2}.$$
By (2), the number of sides for $k$ squares is between $2k+2\sqrt{k}$ and $2k+2\sqrt{k}+1$, thus with fixed perimeter $P$, we have side length $\frac{P}{2k+2\sqrt{k}+1} \leq s \leq \frac{P}{2k+2\sqrt{k}}$.  Hence, the area of an arrangement of $k$ squares ($k\cdot s^2$) satisfies
$$\frac{kP^2}{(2k+2\sqrt{k}+1)^2} \leq A^\ast_k(4) \leq \frac{kP^2}{(2k+2\sqrt{k})^2}.$$
Finally, from (3), the number of sides for an arrangement of $k$ hexagons is less than or equal to $3k+\sqrt{12k-3}+1$, implying that $s \geq \frac{P}{3k+\sqrt{12k-3}+1}$ and therefore,
$$A^\ast_k(6) \geq \frac{6\sqrt{3} \cdot kP^2}{4(3k+\sqrt{12k-3}+1)^2}.$$

Combining these observations, we have $A^\ast_k(3) \leq \frac{\sqrt{3}\cdot kP^2}{(3k+\sqrt{6k})^2} < \frac{kP^2}{(2k+2\sqrt{k}+1)^2} \leq A^\ast_k(4)$ exactly when $\frac{2k+2\sqrt{k}+1}{3k+\sqrt{6k}} < 3^{-1/4}$.  This function of $k$ is decreasing for all $k>0$,  is above $3^{-1/4}$ when $k=4$, and below $3^{-1/4}$ when $k=5$.  Thus, $A^\ast_k(3)< A^\ast_k(4)$ for all $k \geq 5$. 
Moreover, since spiral arrangements of triangles are equivalent (in terms of the number of shared sides) to chains for $k \leq 5$, we also have $A^\ast_k(3) = A_k(3) < A_k(4) \leq A^\ast_k(4)$ for $k=1$, $2$, $3$, and $4$.  Hence, $A^\ast_k(3) < A^\ast_k(4)$ for all $k \geq 1$.

Comparing squares and hexagons, $A^\ast_k(4) \leq \frac{kP^2}{(2k+2\sqrt{k})^2} < \frac{6\sqrt{3} \cdot kP^2}{4(3k+\sqrt{12k-3}+1)^2} \leq A^\ast_k(6)$, exactly when $\frac{3k+\sqrt{12k-3}+1}{k+\sqrt{k}} < 108^{1/4}$.  This function is decreasing for all $k \geq 1$, is above $108^{1/4}$ at $k=6$ and below $108^{1/4}$ at $k=7$.  Thus, $A^\ast_k(4) < A^\ast_k(6)$ for all $k \geq 7$. 
In the previous section, we showed that $A^\ast_k(4) = A_k(4) < A_k(6) = A^\ast_k(6)$ for $k=1$ and $2$ since these are just chains.  For the rest, we can count the number of sides and compare directly: $A^\ast_3(4) = \frac{3}{100}P^2 < \frac{9\sqrt{3}}{450}P^2 = A^\ast_3(6)$, and $A^\ast_4(4) = \frac{1}{36}P^2 < \frac{6\sqrt{3}}{361}P^2 = A^\ast_4(6)$, and $A^\ast_5(4) = \frac{1}{45}P^2 < \frac{15\sqrt{3}}{1058}P^2 = A^\ast_5(6)$, and $A^\ast_6(4) = \frac{6}{289}P^2 < \frac{\sqrt{3}}{81}P^2 = A^\ast_6(6)$.  Hence $A^\ast_k(4) < A^\ast_k(6)$ for all $k \geq 1$ and, to summarize:

\begin{thm}
When given a fixed amount of perimeter to create an optimally packed arrangement of $k$ equal sized pens, the area enclosed by hexagons will be larger than that enclosed by squares, which in turn will be larger than that enclosed by equilateral triangles for all $k \geq 1$.
\end{thm}

\section{Platonic Solid Pens}\label{sec:platonic}
Moving up a dimension, suppose we want to create 3-dimensional enclosures with shared surfaces to save on materials.  Mimicking our 2-dimensional efforts with regular polygons, we might restrict ourselves to using the platonic solids (tetrahedron, cube, octahedron, dodecahedron, and icosahedron).  As above, we'll refer to the total volume of an arrangement of $k$ identical solids using the notation $V_k(f)$ where $f$ counts the number of faces in a single copy.  The formulas for volume and surface area in terms of side length are well-known and allow us to represent the volume of each solid in terms of its surface area, $A$.  In each case, the volume of a single solid is proportional to $A^{3/2}$, thus, for simplicity in what follows we will use $q_f$ to denote the reciprocal of the proportionality constant so that $V_1(f) = \frac{A^{3/2}}{q_f}$ for each solid ($f = 4$, $6$, $8$, $12$, or $20$).  Here are the specific proportionality constants:
$$\text{(tetrahedron)~} q_4 = 6\sqrt{6\sqrt{3}}, \qquad \text{(cube)~} q_6 = 6\sqrt{6}, \qquad \text{(octahedron)~} q_8 = 6\sqrt{3\sqrt{3}}$$
$$\text{(dode)~} q_{12} = \frac{6\sqrt{10-2\sqrt{5}} \cdot (225+90\sqrt{5})^{1/4}}{3+\sqrt{5}},\quad \text{~(ico)~} q_{20} = \frac{12\sqrt{15\sqrt{3}}}{3+\sqrt{5}}$$
Observe that $q_4 > q_6 > q_8 > q_{12} > q_{20}$, thus the icosahedron gives the best ratio of volume to surface area for a single pen, followed by a dodecahedron, octahedron, cube, and tetrahedron.  Just as before though, if we create a chain of $k$ such pens 
with each pen sharing a single face with its neighbors, then there will be a trade off between enclosing more volume and sharing more surface area.

The total surface area for a $k$-length chain is exactly $\frac{kf-(k-1)}{f}A$, where $f$ is the number of faces in the given solid and $A$ is the surface area of a single solid.  
Thus, given a fixed amount of total surface area, $T$, to work with, we may split that surface area up so that each individual solid gets $A = \frac{Tf}{k(f-1)+1}$ (double counting the shared faces).  Using the surface area $A$, we can then calculate the volume of a single solid and scale by $k$ to obtain the total volume for the entire chain.  
For example, with a chain of two tetrahedra ($k=2$, $f=4$), we have $A = \frac{4T}{7}$ and therefore the total volume is $V_2(4) = 2 \cdot \frac{\left(\frac{4T}{7}\right)^{3/2}}{6\sqrt{6\sqrt{3}}} = \frac{8T^{3/2}}{21\sqrt{42\sqrt{3}}}$.

When comparing two different solids with $f$ and $F$ faces respectively, one can rearrange the inequalities as in the previous sections so that:
$$V_k(f) < V_k(F) 
\iff \frac{k(F-1)+1}{k(f-1)+1} < \frac{F}{f}\left(\frac{q_f}{q_F}\right)^{2/3}.$$
This leaves us with a constant on the right of the inequality and a function of $k$ which is increasing for all $k \geq 1$ when $F>f$ and decreasing for all $k \geq 1$ when $F<f$.  Hence, once again, there is at most one value of $k$ for which the two sides are equal.

For example, when comparing the tetrahedron and the cube, we have $V_k(4) < V_k(6)$ if and only if $\frac{5k+1}{3k+1} < \frac{6}{4}\left(\frac{6\sqrt{6\sqrt{3}}}{6\sqrt{6}}\right)^{2/3} = \frac{3^{7/6}}{2}$, which is true for all $k \geq 1$.  On the other hand, $\frac{7k+1}{5k+1} < \left(\frac{4}{3}\right)^{7/6}$ when $k= 67$ and $\frac{7k+1}{5k+1} > \left(\frac{4}{3}\right)^{7/6}$ when $k=68$, so (reporting integer values only) $V_k(6) < V_k(8)$ if and only if $k \leq 67$.  Similar arguments show that $V_k(8) < V_k(12)$ for all $k \geq 1$, that $V_k(6) < V_k(12)$ for all $k \geq 1$, that $V_k(12) < V_k(20)$ if and only if $k \leq 8$, that $V_k(8) < V_k(20)$ for all $k \geq 1$, and that $V_k(6) < V_k(20)$ for all $k \geq 1$.  Hence:

\begin{thm}
When given a fixed amount of surface area to create a chain of $k$ equal sized platonic solid pens, the choices can be ordered by volume contained (from most to least) as follows:
\[\emph{
\begin{tabular}{ccc}
\toprule
\textbf{For} $\bm{k \leq 8}$ & \textbf{For} $\bm{9 \leq k \leq 67}$ & \textbf{For} $\bm{k \geq 68}$\\
\midrule
Icosahedron & Dodecahedron & Dodecahedron\\
Dodecahedron & Icosahedron & Icosahedron\\
Octahedron & Octahedron & Cube\\
Cube & Cube & Octahedron\\
Tetrahedron & Tetrahedron & Tetrahedron\\
\bottomrule
\end{tabular}
}
\]
\end{thm}


 
\begin{rmk}[Spheres]
The volume of a sphere in terms of its surface area is $\frac{A^{3/2}}{6\sqrt{\pi}}$, thus, for a chain of $k$ spheres, we have $A = \frac{T}{k}$ and hence $V_k(\infty) = k \cdot \frac{(\frac{T}{k})^{3/2}}{6\sqrt{\pi}} = \frac{T^{3/2}}{6\sqrt{\pi k}}$.  Using a similar procedure, we observe that 
%
$V_k(12)<V_k(\infty)$ when 
$\frac{k}{11k+1} < \left(\frac{q_{12}}{144\sqrt{3\pi}}\right)^{2/3}$, which is true for all $k \geq 1$, and
%
$V_k(20) < V_k(\infty)$ when 
$\frac{k}{19k+1} < \left(\frac{q_{20}}{240\sqrt{5\pi}}\right)^{2/3}$, which is also true for all $k \geq 1$.  Thus, it is actually more efficient to use spheres -- in spite of the inability to share walls -- than to use any of the platonic solids!
\end{rmk}

\section*{Conclusions and Future Research}
The author's initial question was whether -- regardless of the design/arrangement of the ($n$-dimensional) pens -- it is optimal for the walls in each of the $n$ directions to receive exactly $\frac{1}{n}$ of the allotted boundary.  Corollary~\ref{cor:rect} proves that this is the case for arrangements of rectangular pens for all $n \geq 2$.  In studying the related question of how to maximize the enclosed space using regular polygons and platonic solids, the regularity makes this initial question somewhat nonsensical as the side lengths cannot be adjusted to take better advantage of the shared sides/faces.  If we allow ourselves to consider non-regular shapes however, then the initial question again becomes more relevant.  For example, it is well-known (see e.g.\ \cite{Chakerian}) that for a single $n$-sided shape, the area is maximized when that shape is regular (hence each direction gets an equal amount of the perimeter).  The author has already begun to study the case of multiple, non-regular triangles, where the answer appears to depend on the particular arrangement, and hopes to share this work in a future installment.

\end{document}